\documentclass[11pt,reqno]{amsart}

\usepackage{amsmath,amssymb,amsthm,tikz}
\usepackage[margin=1in]{geometry}
\usepackage{array}
\usepackage{xcolor}
\usepackage{hyperref}
\hypersetup{colorlinks=true, linkcolor=blue!50!black, citecolor=blue!50!black,
  urlcolor=blue!50!black}

\theoremstyle{plain}
\newtheorem{theorem}{Theorem}[section]
\newtheorem{proposition}[theorem]{Proposition}
\newtheorem{lemma}[theorem]{Lemma}
\newtheorem{corollary}[theorem]{Corollary}
\theoremstyle{definition}
\newtheorem{definition}[theorem]{Definition}
\newtheorem{remark}[theorem]{Remark}
\newtheorem{example}[theorem]{Example}

\newcommand{\NN}{\mathbb{N}}
\newcommand{\ZZ}{\mathbb{Z}}
\newcommand{\Per}{\operatorname{Per}}
\newcommand{\Lk}{\operatorname{Lk}}
\newcommand{\Vm}{V_m}
\newcommand{\Sn}{S_n}
\newcommand{\eps}{\varepsilon}
\newcommand{\axis}{\operatorname{axis}}

\begin{document}

\title[Anti-tori from full Mealy automata]
{Full Mealy automata, complete square complexes, and anti-tori}

\author{David Pask}
\date{\today}
\address{Math Department, James Cook University, 1 James Cook Dr, Douglas QLD 4814, Australia}
\email{david.a.pask@gmail.com}

\begin{abstract}
To a full $m\times n$ Mealy automaton $A$ we associate a bijection $\theta_A$, a
one-vertex rank-two graph $F_{\theta_A}$, and a one-vertex $VH$-square complex
$Y_A$ tiled by $mn$ Wang tiles. We prove that $Y_A$ contains an anti-torus if and
only if $A$ is bi-reversible and $F_{\theta_A}$ is aperiodic. The two hypotheses
are independent and play disjoint roles: bi-reversibility is exactly what makes
$Y_A$ a complete square complex, so that its universal cover splits as a product
of two trees and anti-tori can be discussed at all; and, within that setting, an
anti-torus is precisely a period-free configuration in the two-sided path space of
$F_{\theta_A}$, whose existence is the aperiodicity condition. Working at the level
of configurations removes any appeal to the geometry of products of trees from the
main equivalence; the geometric (loop-spanned) form of Wise is shown to be strictly
stronger, the lamplighter being aperiodic with no loop-spanned anti-torus.
\end{abstract}

\subjclass[2010]{Primary: {20F65,46L05}; Secondary: {20E08,20F67,37B50,37B10,20M35}}
\keywords{Single vertex rank $2$-graph, Mealey automata, anti-torus, complete square complex, Burger-Mozes groups}

\maketitle

\section{Background and motivation}

This paper grew out of the author's interests in $C^*$-algebras (in particular
higher-rank graphs) and in dynamical systems (in particular shift spaces).
$C^*$-algebraists, always in search of new examples, have studied the operator
algebras associated with products of trees and with $\tilde{A}_1\times\tilde{A}_1$
buildings \cite{RRS,KR}, as well as the cubical complexes associated with
single-vertex $k$-graphs \cite{LV,MRV}.

The attention paid to Burger--Mozes groups \cite{BM1,BM2} in this circle of ideas
led the author to the work of Glasner and Mozes \cite{GM}, in which Mealy automata
play a central role, and from there to \cite{KlPi}. One of the themes of \cite{KlPi}
is the relationship between Mealy automata and Wang tiles, a connection that
resonates with the author's earlier work on textile systems and the Wang tiles
arising from $2$-graphs.

The literature on complete square complexes led in turn to \cite{W} and \cite{Rat} who deal with the
notion of an \emph{anti-torus}. The example \cite[Example~4.1]{W} makes the
connection precise: it is a tiling arising from an automaton, and it becomes clear
that further examples arise from $2$-graphs. Determining exactly which $2$-graphs do
so is the question answered here, and the answer sets the scene for the
higher-dimensional examples we intend to pursue in future work.

\section{Introduction}

A full Mealy automaton over alphabets of sizes $m$ and $n$ is the same data as a
bijection $\theta\colon V_m\times S_n\to S_n\times V_m$, equivalently a one-vertex
rank-two graph $F_\theta$ \cite{KP,Y}, equivalently a set of $mn$ Wang tiles closed
under the factorisation rules \cite{HRSW,RSY}. The same data presents a square
complex with a single vertex, $m$ vertical and $n$ horizontal loops, and $mn$
squares. When this complex is a \emph{complete square complex} (CSC) its universal
cover is a product of two trees \cite{W,GM}, and one asks whether it contains an
\emph{anti-torus}: a plane whose induced tiling is not periodic \cite[\S5]{W}.

\begin{theorem}\label{thm:main}
Let $A=(\Vm,\Sn,\pi,\lambda)$ be a full Mealy automaton with bijection $\theta_A$,
rank-two graph $F:=F_{\theta_A}$, and $VH$-square complex $Y:=Y_A$. Then $Y$
contains an anti-torus if and only if $A$ is bi-reversible and $F$ is aperiodic.
\end{theorem}

The proof separates the two hypotheses. Bi-reversibility is equivalent to four-way
determinism of the Wang tiles, hence to the link of the vertex being complete
bipartite, hence to $Y$ being a CSC (Proposition~\ref{prop:csc}). Granted that,
the planes of the universal cover carry exactly the two-sided paths of $F$, with
periods matching (Proposition~\ref{prop:planes}); an anti-torus is then a
period-free configuration (Lemma~\ref{lem:config}), whose existence is aperiodicity
(Proposition~\ref{prop:antitorus}). The geometric form of Wise, an anti-torus
spanned by two loop axes, is strictly stronger and is not implied by aperiodicity
alone (Proposition~\ref{prop:loops}).

\section{Preliminaries}\label{sec:prelim}

A Mealy automaton $A=(\Vm,\Sn,\pi,\lambda)$ has $\pi\colon\Vm\times\Sn\to\Vm$ and
$\lambda\colon\Vm\times\Sn\to\Sn$. It is \emph{full} if
$\pi\times\lambda\colon(u,b)\mapsto(\lambda(u,b),\pi(u,b))$ is a bijection
$\Vm\times\Sn\to\Sn\times\Vm$; \emph{invertible} if each $\lambda(u,\cdot)$ is a
bijection of $\Sn$; \emph{reversible} if each $\pi(\cdot,b)$ is a bijection of
$\Vm$; \emph{bi-reversible} if it is invertible, reversible, and its inverse is
reversible. 
In this context, a Mealy automaton $A$ is a complete, deterministic,
letter-to-letter (that is, synchronous) transducer with the same input and output
alphabet; see \cite{KlPi,GNS} for further details.

Each pair $(u,b) \in V_m \times S_n$ gives a Wang tile $T_{(u,b)}$ with
$N=\lambda(u,b)$, $S=b$, $E=u$, $W=\pi(u,b)$, north/south coloured by letters and
east/west by states. We freely use the determinism dictionary
\cite{GM}: a full $A$ is se-deterministic, and $W_A$ is
ne-deterministic $\iff$ reversible, sw-deterministic $\iff$ invertible, four-way
deterministic $\iff$ bi-reversible.

\begin{definition}[One-vertex rank-two graph; cf.\ \cite{Y},
\cite{KP}]\label{def:Ftheta}
Let $\theta\colon\Vm\times\Sn\to\Sn\times\Vm$ be a bijection, written
$\theta(u,b)=(\lambda(u,b),\pi(u,b))$, and let $\eps_1,\eps_2$ denote the standard
generators of $\NN^2$. The associated \emph{one-vertex rank-two graph} $F_\theta$ is
the monoid
\[
F_\theta=\big\langle\, \Vm\sqcup\Sn \ \big|\ u\,b=\lambda(u,b)\,\pi(u,b)
\ \text{for all } (u,b)\in\Vm\times\Sn \,\big\rangle,
\]
with identity $\ast$, equipped with the degree functor $d\colon F_\theta\to\NN^2$
determined by $d(\ast)=(0,0)$, $d(u)=\eps_1$ for $u\in\Vm$ and $d(b)=\eps_2$ for
$b\in\Sn$ (so $d(w)$ records the numbers of $\Vm$- and $\Sn$-letters in any word
representing $w$). Viewing $F_\theta$ as a category with the single object $\ast$,
the pair $(F_\theta,d)$ is a rank-two graph: for every $w\in F_\theta$ and all
$p,q\in\NN^2$ with $d(w)=p+q$ there are unique $w',w''\in F_\theta$ with $d(w')=p$,
$d(w'')=q$ and $w=w'w''$. The commuting square through each pair $(u,b)$ is the
factorisation $u\,b=\lambda(u,b)\,\pi(u,b)$, that is, the Wang tile $T_{(u,b)}$.
\end{definition}

\begin{remark}\label{rem:onevertex}
Bijectivity of $\theta$ is exactly the factorisation property, and conversely every
rank-two graph with a single vertex is $F_\theta$ for a unique such $\theta$. For a
full automaton $A$ we write $F:=F_{\theta_A}$, where $\theta_A=\pi\times\lambda$; the
hypothesis that $A$ be full is precisely what makes $F_{\theta_A}$ a rank-two graph. Let $W_A$ denote the set of tiles $T_{(u,b)}$ where $(u,b) \in V \times S$, indexed by $\mathcal{C}_{T_A}$.
\end{remark}

\noindent
Let $\Delta$ denote the rank-two graph with object set $\ZZ^2$, morphisms
$\{(p,q)\in\ZZ^2\times\ZZ^2:p\le q\}$, and degree map $d(p,q)=q-p$ (see \cite{KP}).

\begin{definition}
Let $\Lambda$ be a rank-two graph. The \emph{two-sided infinite path space} of
$\Lambda$ is
\[
\Lambda^\Delta ~=~ \{\, x:\Delta\to\Lambda : x \text{ is a degree-preserving
functor}\,\}.
\]
\end{definition}

Following \cite{KP}, the two-sided path space $F^\Delta$ of $F=F_{\theta_A}$ is
identified with the set of tilings of $\ZZ^2$ by $W_A$ respecting the
factorisation rules. For $x\in F^\Delta$ a \emph{period} is a $p\in\ZZ^2$ with
$x(\cdot+p,\cdot+p)=x$; the periods form a subgroup $\Per(F,x)\le\ZZ^2$, and we set
$\Per(F)=\bigcap_x\Per(F,x)$.

\begin{definition}\label{def:aper}
$F$ is \emph{aperiodic} if some $x\in F^\Delta$ has $\Per(F,x)=\{0\}$; otherwise it
is \emph{periodic}.
\end{definition}

\begin{remark}[Periodic dichotomy] \label{rmk:aperiod}
By Davidson--Yang \cite[Theorem~3.1]{DY}, exactly one of
the following holds: either a single nonzero $(a,-b)$ lies in $\Per(F,x)$ for every
$x$ (periodic), or some configuration is period-free (aperiodic). In the periodic
case the period is implemented by a bijection $\gamma\colon F^{a\eps_1}\to
F^{b\eps_2}$ with $\mu\nu=\gamma(\mu)\gamma^{-1}(\nu)$.
\end{remark}

\noindent
The following result appears in \cite[Corollary 4.1]{DY}.

\begin{proposition}[Aperiodicity criterion]\label{prp:aperiod}
Write $\theta(v_i,s_j)=(\alpha_{v_i}(s_j),\beta_{s_j}(v_i))$, where
$\alpha_{v_i}\colon S_n\to S_n$ and $\beta_{s_j}\colon V_m\to V_m$. If there is a
subset $B\subseteq S_n$ with $|B|\ge2$ and a word $i_1\ldots i_k\in V_m^k$ such that
$\alpha_{i_1}\circ\alpha_{i_2}\circ\cdots\circ\alpha_{i_k}(B)=B$, then $F_\theta$ is
aperiodic. Similarly, if there is a subset $A\subseteq V_m$ with $|A|\ge2$ and a word
$j_1\ldots j_k\in S_n^k$ such that
$\beta_{j_1}\circ\beta_{j_2}\circ\cdots\circ\beta_{j_k}(A)=A$, then $F_\theta$ is
aperiodic.
\end{proposition}

\begin{proof}
We prove the contrapositive of each statement. Suppose $F_\theta$ is not aperiodic.
By Remark~\ref{rmk:aperiod} it then has a period $(a,-b)$ with
$(a,b)\neq(0,0)$, so $x(m+a\eps_1,\,n+b\eps_2)=x(m,n)$ for every $x\in F_\theta^\Delta$
and every $(m,n)\in\Delta$. Reading this along the two axes shows that for every
$j\in S_n$ and every word $i_1\ldots i_k\in V_m^k$,
\[
\alpha_{i_1}\circ\cdots\circ\alpha_{i_k}(j)=j,
\]
and for every $i\in V_m$ and every word $j_1\ldots j_k\in S_n^k$,
\[
\beta_{j_1}\circ\cdots\circ\beta_{j_k}(i)=i.
\]
Thus no proper composition of the $\alpha$'s or $\beta$'s can fix a subset of size at
least two nontrivially, which is the required contrapositive.
\end{proof}

A one-vertex $VH$-square complex $Y$ has subcomplexes $\mathcal V_Y,\mathcal H_Y$ of
vertical and horizontal edges; its link $\Lk(*)$ is bipartite with $2m$ vertical and
$2n$ horizontal germs, one edge per square corner. $Y$ is a \emph{complete square
complex} (CSC) if $\Lk(*)$ is complete bipartite; then $\widetilde Y\cong
\mathcal V_{\widetilde Y}\times\mathcal H_{\widetilde Y}$ is a product of two trees
\cite[Theorem~3.8]{W}. We attach to $A$ the complex $Y_A$ with vertical loops $\Vm$,
horizontal loops $\Sn$, and one square per tile $T_{(u,b)}$.

\begin{definition}
Let $E,F$ be directed graphs and let $p,q\colon F\to E$ be surjective graph
morphisms. We call $T=(p,q\colon F\to E)$ a \emph{textile system} if the map
$\iota\colon F^1\to E^1\times E^1\times F^0\times F^0$,
$\iota(f)=(p(f),q(f),r(f),s(f))$, is injective.
\end{definition}

\noindent
Let $p\colon F\to E$ be a graph morphism. We say that $p$ has \emph{unique $r$-path
lifting} (respectively, \emph{unique $s$-path lifting}) if for every $e\in p(F^1)$
and every $w\in p^{-1}(r(e))$ (respectively, $w\in p^{-1}(s(e))$) there is a unique
$f\in wF^1$ (respectively, $f\in F^1w$) with $p(f)=e$.

\begin{definition}[Textile system from a Mealy automaton]
Let $A=(V,S,\pi,\lambda)$ be a Mealy automaton, with $\pi\colon V\times S\to V$ and
$\lambda\colon V\times S\to S$. Define directed graphs $E,F$ by
\[
E^0=\{*\},\quad E^1=V,\quad F^0=S,\quad F^1=V\times S,
\qquad r(u,b)=\lambda(u,b),\ \ s(u,b)=b,
\]
and graph morphisms $p,q\colon F\to E$ by
\[
p(u,b)=\pi(u,b),\qquad q(u,b)=u \qquad\text{for } (u,b)\in V\times S=F^1.
\]
Since the map
\[
\iota\colon (u,b)\longmapsto (p(u,b),r(u,b),s(u,b),q(u,b))=(\pi(u,b),\lambda(u,b),b,u)
\]
is injective, $T_A=(p,q\colon F\to E)$ is a textile system, with squares
$\mathcal{C}_{T_A}$ of the form
\begin{equation} \label{eq:textile}
\begin{array}{l}
\begin{tikzpicture}[>=stealth,xscale=2,yscale=2]
\draw (0,0)--(1,0)--(1,1)--(0,1)--(0,0)--(1,1)--(0,1)--(1,0);

\draw[dotted] (0,0)--(-0.15,-0.15);
\draw[dotted] (1.15,1.15)--(1,1);
\node at (1.2,1.15) {\Tiny$y=x$};

\node at (-0.65,0.5) {\tiny $T_{(u,b)}:=$};

\node at (0.52,0.85) {$\scriptstyle \pi(u,b)$};
\node at (0.5,0.15) {$\scriptstyle u$};
\node at (0.2,0.5) {\tiny $\lambda(u,b)$};
\node at (0.8,0.5) {\tiny $b$};

\end{tikzpicture}
\end{array}
\end{equation}
\end{definition}

\noindent
Our conventions differ from those of \cite[\S10.2]{KlPi}: we have adapted the
orientation so as to make the correspondence with rank-two graphs more transparent,
with the effect that our tile is the reflection of theirs in the line $y=x$ (see
\eqref{eq:textile}). We use the $N,E,S,W$ labelling of the four edges of a tile and
the notions of $NE$-determinism, and so on, following \cite[p.~398]{KlPi}.

\section{Bi-reversibility is completeness of the square complex}

\begin{proposition}\label{prop:csc}
Let $A$ be full. Then $Y_A$ is a complete square complex if and only if $A$ is
bi-reversible.
\end{proposition}

\begin{proof}
$\Lk(*)$ is bipartite with parts the $2m$ vertical germs $\{u^\pm\}$ and the $2n$
horizontal germs $\{b^\pm\}$; the four corners of $T_{(u,b)}$ contribute, one to
each sign pattern, the (north,east), (south,east), (north,west), (south,west)
colour pairs. Thus $\Lk(*)=K_{2m,2n}$ if and only if every (vertical germ,
horizontal germ) pair occurs as exactly one corner, if and only if for each of the
four corner positions the map $T_{(u,b)}\mapsto$ (its two colours there) is a
bijection from the $mn$ tiles onto the $mn$ possible pairs, if and only if $W_A$ is
four-way deterministic, if and only if $A$ is bi-reversible \cite{GM}. Equivalently,
the squares $\mathcal C_{T_A}$ are complete (in the sense of \cite{HRSW}) if and only
if the textile maps $p,q$ have unique $r$- and $s$-path lifting.
\end{proof}

If $A$ is full but not bi-reversible, $Y_A$ is not a CSC, $\widetilde Y_A$ is not a
product of trees, and the planes below are undefined; such a $Y_A$ has no
anti-torus for want of the ambient product geometry.

\section{Configurations, planes, and anti-tori}

Assume henceforth that $A$ is full and bi-reversible, so $\widetilde Y_A\cong
\mathcal V\times\mathcal H$ (cf. \cite{GM, W}).

\begin{proposition}\label{prop:planes}
There is a bijection between the planes $\gamma_H\times\gamma_V$ of $\mathcal
V\times\mathcal H$ spanned by a bi-infinite horizontal and vertical geodesic and the
configurations $x\in F^\Delta$, under which the induced square tiling of the plane
is $x$, and a translation by $p\in\ZZ^2$ preserves the tiling if and only if
$p\in\Per(F,x)$.
\end{proposition}

\begin{proof}
A plane in the product of trees is a copy of $\ZZ^2$ tiled by lifts of the squares;
reading the vertical and horizontal edges crossed gives bi-infinite words in $\Vm$,
$\Sn$, and the factorisation rules of $F$ are the local compatibility relations of
the tiles, so the tiling is a degree-preserving functor $x\colon\Delta\to F$, and
every $x$ arises. A grid translation by $p$ fixes the tiling if and only if
$x(\cdot+p,\cdot+p)=x$, that is, $p\in\Per(F,x)$.
\end{proof}

\begin{definition}\label{def:antitorus}
A plane in $\widetilde Y_A$ is an \emph{anti-torus} if its induced tiling is not
periodic, i.e.\ has no nonzero period. We say $Y_A$ \emph{contains an anti-torus}
if some plane is one. (A fortiori such a plane does not factor through a torus
$\mathbb T^2\to Y_A$.)
\end{definition}

\begin{lemma}[Configuration model of anti-tori]\label{lem:config}
A plane of $\widetilde Y_A$ is an anti-torus if and only if its configuration
$x\in F^\Delta$ has $\Per(F,x)=\{0\}$.
\end{lemma}

\begin{proof}
This is immediate from Proposition~\ref{prop:planes}: the plane has a nonzero period if and only if
$\Per(F,x)\neq\{0\}$.
\end{proof}

Lemma~\ref{lem:config} is the point of the present treatment: it turns the
geometric notion of an anti-torus into the purely combinatorial statement
$\Per(F,x)=0$, with no further appeal to the structure theory of products of trees.

\begin{proposition}\label{prop:antitorus}
With $A$ full and bi-reversible, $Y_A$ contains an anti-torus if and only if the
rank-two graph $F$ is aperiodic.
\end{proposition}

\begin{proof}
By Lemma~\ref{lem:config} a plane is an anti-torus if and only if its configuration is
period-free, and by Proposition~\ref{prop:planes} every configuration occurs.
Hence $Y_A$ contains an anti-torus if and only if some $x\in F^\Delta$ has $\Per(F,x)=\{0\}$,
which is Definition~\ref{def:aper}.
\end{proof}

\section{The Wise loop-spanned form}\label{sec:loops}

Wise's anti-tori \cite{W} are built from a pair of loops. Call a plane of
$\widetilde Y_A$ \emph{loop-spanned} if both its geodesics are axes of loops, i.e.\
if it is $\axis(a)\times\axis(b)$ for a vertical loop $a$ and a horizontal loop $b$;
write $x_{a,b}$ for its configuration, whose central cross is the periodic pair
$(a^\infty,b^\infty)$. A \emph{Wise anti-torus} is a loop-spanned plane that is
period-free --- equivalently, a pair $(a,b)$ no positive powers of which commute. One
might hope aperiodicity supplies such a plane; it does not.

\begin{proposition}\label{prop:loops}
A Wise anti-torus is an anti-torus in the sense of Definition~\ref{def:antitorus};
hence if $Y_A$ has a Wise anti-torus then $A$ is bi-reversible and $F$ is aperiodic.
The converse is \emph{false}: there are full bi-reversible $A$ with $F$ aperiodic for
which every loop-spanned plane is periodic, so $Y_A$ has no Wise anti-torus.
\end{proposition}

\begin{proof}
A loop-spanned period-free plane is in particular period-free, hence an anti-torus,
so $A$ is bi-reversible and $F$ aperiodic by Theorem~\ref{thm:main}. For the failure
of the converse take the lamplighter $\theta_4^a$ of Appendix~\ref{app:2x2}, which is
bi-reversible with $F$ aperiodic. Since $\lambda(v_i,\cdot)\equiv\mathrm{id}$, the
north and south labels of every tile agree, so the horizontal label is constant up
each column, say $b_i$, and the states propagate horizontally,
$a_{i-1,j}=\pi(a_{i,j},b_i)$, giving $a_{i,j}=\sigma^{c_i}(a_{0,j})$ for suitable
$c_i\in\{0,1\}$; choosing the vertical word $(a_{0,j})_j$ aperiodic gives
$\Per(F,x)=\{0\}$, so $F$ is aperiodic. But a loop-spanned plane has $(a_{0,j})_j$ a
periodic axis, which forces a vertical period; so no loop-spanned plane of
$\theta_4^a$ is period-free.
\end{proof}

\begin{remark}\label{rem:gap}
The two notions therefore genuinely differ, the obstruction being the
\emph{realisation} of a period-free configuration on loop axes. Four-way determinism
makes the loop-spanned configurations dense in $F^\Delta$, and the Davidson--Yang
dichotomy given in Remark~\ref{rmk:aperiod} controls which periods occur once $F$
\emph{is} periodic; but density together with that control does not produce a
period-free loop-spanned plane --- the periodic loop-spanned planes of $\theta_4^a$
need not share a single period, so they impose none on $F$. The genuinely
two-dimensional, free behaviour that does carry a Wise
anti-torus first appears at $m=2$, $n=3$ (Example~\ref{ex:53}), whose tiling has both
axes constant, hence loop-spanned. Characterising the aperiodic $F$ that admit a Wise
anti-torus is exactly the question left to \cite{W,GM}.
\end{remark}

\subsection*{Relation to square-complex groups}
The loop-spanned form is the point of contact with the group-theoretic theory of
anti-tori in $(2m,2n)$--groups developed by Rattaggi \cite{Rat}, following
Bridson--Wise and Wise \cite{W}.

\begin{remark}[The loop-spanned form is Rattaggi's group-theoretic anti-torus]
\label{rem:rattaggi}
When $A$ is full and bi-reversible, $Y_A$ is a complete square complex and
$\pi_1(Y_A)$ is a $(2m,2n)$--group in the sense of \cite[\S1]{Rat}: its link is the
complete bipartite graph $K_{2m,2n}$ (Proposition~\ref{prop:csc}), which is exactly
Rattaggi's link condition, and the $mn$ squares\,---\,whose boundaries alternate
between $\Vm$-edges and $\Sn$-edges\,---\,are his relator set $R_{m\cdot n}$. The two
factors $\langle\Vm\rangle$ and $\langle\Sn\rangle$ are then free of ranks $m$ and
$n$ \cite[Corollary~2(1)]{Rat}, and $\pi_1(Y_A)$ is centreless once $m,n\ge2$
\cite[Corollary~2(2)]{Rat}. Under this dictionary a Wise anti-torus, spanned by the
axes of $a\in\langle\Vm\rangle$ and $b\in\langle\Sn\rangle$, is precisely an
anti-torus in the sense of \cite[Definition~4]{Rat}: a pair with
$a^{r}b^{s}\neq b^{s}a^{r}$ for all $r,s\in\ZZ\setminus\{0\}$ (equivalently, no
positive powers of $a,b$ commute). Thus Rattaggi's definition attaches to the
\emph{stronger}, loop-spanned notion of this section\,---\,not to the configuration
anti-torus of Definition~\ref{def:antitorus}, which has no direct group-theoretic
counterpart and is the genuinely new object isolated by Lemma~\ref{lem:config}.
\end{remark}

\begin{remark}[Rattaggi's centralizer criterion in automaton terms]\label{rem:rho}
Rattaggi's homomorphism $\rho_v$ from one free factor to the symmetric group on the
generators of the other \cite[\S1]{Rat}, defined on generators by
$a^{-1}ba'=\tilde b$, is on single generators our corner permutation: it is the map
$\beta_{b}\colon\Vm\to\Vm$ of Proposition~\ref{prp:aperiod} (and the transposed
homomorphism recovers the $\alpha$'s). His Lemma~3 then reads: if some generator
$b$ acts by a fixed-point-free corner permutation\,---\,$\beta_{b}(a)\neq a$ for all
$a$\,---\,then $Z_{\pi_1(Y_A)}(b)=\langle b\rangle\cong\ZZ$
\cite[Lemma~3]{Rat}. This is a finite, tile-level check on a single corner map.

We stress, however, that the passage from a cyclic centralizer to an anti-torus
\emph{for every} $a\neq1$ \cite[Corollary~8]{Rat}, and the dichotomy ``anti-torus
$\iff$ non-commuting'' \cite[Proposition~5]{Rat}, both require $\pi_1(Y_A)$ to be
\emph{commutative transitive}. Rattaggi's quaternionic groups $\Gamma_{p,l}$ are
commutative transitive \cite[Proposition~14]{Rat}, but a generic bi-reversible full
automaton group is not, and we do not assume it. Lemma~3 transfers unconditionally
and yields the centralizer collapse; the anti-torus conclusion does not. This is
precisely why the configuration/aperiodicity route of Theorem~\ref{thm:main}, rather
than the group-theoretic dichotomy, is what detects an anti-torus in general: in
Rattaggi's commutative-transitive world the configuration-level subtlety isolated by
Lemma~\ref{lem:config} does not arise, whereas here it is unavoidable.
\end{remark}

\begin{corollary}[Irreducibility]\label{cor:irred}
If $Y_A$ contains a Wise anti-torus, then $\pi_1(Y_A)$ is an irreducible lattice in
$\mathrm{Aut}(\mathcal T_{2m})\times\mathrm{Aut}(\mathcal T_{2n})$.
\end{corollary}

\begin{proof}
A Wise anti-torus is an anti-torus in $\pi_1(Y_A)$ in the sense of
\cite[Definition~4]{Rat} (Remark~\ref{rem:rattaggi}); apply
\cite[Proposition~9]{Rat}, due to Wise.
\end{proof}

\section{Proof of the characterisation}

\begin{proof}[Proof of Theorem~\ref{thm:main}]
If $Y_A$ contains an anti-torus then a plane of $\widetilde Y_A$ is defined, so
$Y_A$ is a CSC and $A$ is bi-reversible (Proposition~\ref{prop:csc}); and then $F$
is aperiodic (Proposition~\ref{prop:antitorus}). Conversely, if $A$ is
bi-reversible then $Y_A$ is a CSC (Proposition~\ref{prop:csc}), and if $F$ is
aperiodic then $Y_A$ contains an anti-torus (Proposition~\ref{prop:antitorus}).
\end{proof}

\begin{corollary}\label{cor:simple}
$Y_A$ contains an anti-torus if and only if $A$ is bi-reversible and
$C^*(F_{\theta_A})$ is simple.
\end{corollary}

\begin{proof}
For these row-finite cofinal one-vertex rank-two graphs, simplicity of
$C^*(F_{\theta_A})$ is the aperiodicity condition \cite{RobS,DY}; combine with
Theorem~\ref{thm:main}.
\end{proof}

\section{A worked anti-torus, and the role of each hypothesis}

\begin{example}[A bi-reversible aperiodic automaton: an anti-torus]\label{ex:53}
Let $V=\{f_0,f_1\}$ and $S=\{g_0,g_1,g_2\}$. This realises the automaton of
\cite[Example~4.1]{W} (the transpose of the one in Appendix~\ref{app:3x2}), with six
tiles $T_{(u,b)}=(N,S,E,W)=(\lambda(u,b),b,u,\pi(u,b))$:
\[
\begin{array}{c|cccc}
T_{(u,b)} & N & S & E & W\\\hline
T_{(f_0,g_0)} & g_1 & g_0 & f_0 & f_0\\
T_{(f_0,g_1)} & g_2 & g_1 & f_0 & f_1\\
T_{(f_0,g_2)} & g_0 & g_2 & f_0 & f_1\\
T_{(f_1,g_0)} & g_1 & g_0 & f_1 & f_1\\
T_{(f_1,g_1)} & g_0 & g_1 & f_1 & f_0\\
T_{(f_1,g_2)} & g_2 & g_2 & f_1 & f_0
\end{array}
\]
Each of the four corner maps is a bijection onto its six-element range; e.g.\ the
(south,west) pairs are
$(g_0,f_0),(g_1,f_1),(g_2,f_1),(g_0,f_1),(g_1,f_0),(g_2,f_0)$, all distinct.
So the tiles are four-way deterministic and, by Proposition~\ref{prop:csc}, $Y$ is
a complete square complex.

The output permutations are $\alpha_{f_0}=(g_0\,g_1\,g_2)$ and
$\alpha_{f_1}=(g_0\,g_1)$, which generate a transitive, non-cyclic action. By the
argument in Appendix~\ref{app:3x2}, the two-sided path space carries an aperiodic
configuration and the associated self-similar group is the free group $\mathbb{F}_3$
\cite{BGKMNSS}; in particular $F$ is aperiodic. By Theorem~\ref{thm:main}, $Y$
contains an anti-torus.

Concretely, propagating the unique tiling with all-$f_1$ left column and all-$g_2$
bottom row via the $(\text{south,west})\to$ tile rule, the first two tile-columns are
\[
\text{col }1:\ T_{(f_0,g_2)},\,T_{(f_1,g_0)},\,T_{(f_0,g_1)},\,T_{(f_0,g_2)},\dots
\quad(\text{vertical period }3),
\]
\[
\text{col }2:\ T_{(f_1,g_2)},\,T_{(f_0,g_2)},\,T_{(f_0,g_0)},\,T_{(f_1,g_1)},\,
T_{(f_1,g_0)},\,T_{(f_1,g_1)},\dots,
\]
with north-edge strings $g_0g_1g_2g_0\cdots$ in column $1$ and $g_2g_0g_1g_0g_1\cdots$
in column $2$: no common vertical period. The non-commutation is already visible in
$f_0g_0=g_1f_0\neq g_0f_0$, and $f_0$ acts with infinite order:
$f_0\!\cdot\!g_0,f_0^2\!\cdot\!g_0,f_0^3\!\cdot\!g_0$ cycle through $g_1,g_2,g_0$,
yet $f_0^3\!\cdot\!(g_0g_0)=g_0g_1\neq g_0g_0$, so $f_0^3\neq\mathrm{id}$. This
free-like behaviour is the mechanism of the anti-torus.
\end{example}

\begin{example}[Full but not four-way deterministic: no complete square complex]
\label{ex:lamp}
The automaton $\theta_7$ of Table~\ref{tables} is full --- it presents
a one-vertex rank-two graph --- but is neither invertible nor reversible: both
$\lambda(v_0,\cdot)$ and $\lambda(v_1,\cdot)$ are constant, and so is
$\pi(\cdot,s_0)$, so none is a bijection. Hence it is not four-way deterministic, $Y$
is not a complete square complex (Proposition~\ref{prop:csc}), the universal cover is
not a product of two trees, and there is no anti-torus. This shows the
bi-reversibility hypothesis of Theorem~\ref{thm:main} secures the ambient geometry
and cannot be dropped. By contrast the lamplighter $\theta_4^a$ is full \emph{and}
four-way deterministic --- in the $2$-graph presentation of
Table~\ref{tables} it is invertible and reversible --- and is
the aperiodic anti-torus treated in Appendix~\ref{app:2x2}.
\end{example}

\begin{example}[Bi-reversible but periodic: no anti-torus]
The automaton $\theta_2$ of Table~\ref{tables} is bi-reversible, so
$Y_{\theta_2}$ is a CSC, but $F_{\theta_2}$ is periodic, with
$C^*(F^+_{\theta_2})\cong M_2(C(\mathbb T))$. By Theorem~\ref{thm:main} (or
Corollary~\ref{cor:simple}, since $M_2(C(\mathbb T))$ is not simple) it contains no
anti-torus: every plane factors through a torus. The trivial $\theta_1$, with
$F_{\theta_1}\cong\NN^2$ and $C^*\cong C(\mathbb T^2)$, is similar --- every plane
\emph{is} a torus.
\end{example}

\section{Remarks}

\begin{remark}[The hypotheses are independent]
The bijection $\theta_2$ in Table~\ref{tables} is bi-reversible but periodic, so aperiodicity cannot be dropped. In the
other direction, a full automaton that is not four-way deterministic --- e.g.\
$\theta_7$ of Example~\ref{ex:lamp} --- does not
present a complete square complex, so its ambient space is not a product of two trees
and bi-reversibility cannot be dropped either. The two requirements are logically
independent; at $2\times2$ they are realised separately, the only aperiodic
$2\times2$ automaton (the lamplighter) being itself bi-reversible, so that the
genuinely mixed case appears only for larger alphabets.
\end{remark}

\begin{remark}[Configuration versus loop form]
Lemma~\ref{lem:config} recasts an anti-torus (Definition~\ref{def:antitorus}) as a
period-free configuration, so the main equivalence (Theorem~\ref{thm:main}) is
internal to the rank-two graph, with no appeal to \cite{W,GM}. The geometric
loop-spanned form of \cite{W} is strictly stronger
(Proposition~\ref{prop:loops}) and is \emph{not} equivalent to aperiodicity --- the
lamplighter (Appendix~\ref{app:2x2}) is aperiodic with no Wise anti-torus --- so the
loop-realisation step is exactly what still rests on \cite{W,GM}.
\end{remark}

\begin{remark}[Higher rank]
The template suggests a rank-$d$ statement: a full $d$-dimensional automaton with
$2^d$-way determinism presents a one-vertex rank-$d$ graph and a cube complex whose
universal cover is a product of $d$ trees, and ``contains a $d$-dimensional
anti-torus'' should be full determinism together with aperiodicity of the rank-$d$
graph. The determinism half is routine; the periodicity half needs the rank-$d$
analogue of \cite{DY}, and the cube complex requires an associativity constraint
absent at $d=2$.
\end{remark}

\appendix

\section{The full \texorpdfstring{$2\times2$}{2x2} automata}\label{app:2x2}

Below we list the nine full Mealy automata with $|V|=|S|=2$,
labelled $\theta_1,\dots,\theta_8$ (with $\theta_4$ occurring as $\theta_4^a$ and
$\theta_4^c$). They are a natural test of Theorem~\ref{thm:main}: we settle the
anti-torus question for each through the two independent filters, bi-reversibility
(Proposition~\ref{prop:csc}) and aperiodicity (Proposition~\ref{prop:antitorus}).

\emph{Multiplication tables.} The cell in row $v_i$, column $s_j$ records the
right-hand side $s_{j'}v_{i'}$ of the factorisation $v_is_j=s_{j'}v_{i'}$, equivalently
$\theta(v_i,s_j)=(s_{j'},v_{i'})$.
\begin{table}[ht]
\centering
\renewcommand{\arraystretch}{1.15}
\begin{tabular}{ccc} 
$\theta_1\colon\begin{array}{c|cc} & s_0 & s_1\\\hline v_0 & s_0v_0 & s_0v_1\\ v_1 & s_1v_0 & s_1v_1\end{array}$ &
$\theta_2\colon\begin{array}{c|cc} & s_0 & s_1\\\hline v_0 & s_1v_0 & s_0v_0\\ v_1 & s_0v_1 & s_1v_1\end{array}$ &
$\theta_3\colon\begin{array}{c|cc} & s_0 & s_1\\\hline v_0 & s_1v_1 & s_1v_0\\ v_1 & s_0v_1 & s_0v_0\end{array}$ \\[2ex]
$\theta_4^a\colon\begin{array}{c|cc} & s_0 & s_1\\\hline v_0 & s_0v_1 & s_1v_0\\ v_1 & s_0v_0 & s_1v_1\end{array}$ &
$\theta_4^c\colon\begin{array}{c|cc} & s_0 & s_1\\\hline v_0 & s_0v_0 & s_1v_1\\ v_1 & s_0v_1 & s_1v_0\end{array}$ &
$\theta_5\colon\begin{array}{c|cc} & s_0 & s_1\\\hline v_0 & s_0v_1 & s_0v_0\\ v_1 & s_1v_1 & s_1v_0\end{array}$ \\[2ex]
$\theta_6\colon\begin{array}{c|cc} & s_0 & s_1\\\hline v_0 & s_1v_1 & s_1v_0\\ v_1 & s_0v_1 & s_0v_0\end{array}$ &
$\theta_7\colon\begin{array}{c|cc} & s_0 & s_1\\\hline v_0 & s_1v_0 & s_0v_0\\ v_1 & s_1v_1 & s_0v_1\end{array}$ &
$\theta_8\colon\begin{array}{c|cc} & s_0 & s_1\\\hline v_0 & s_0v_1 & s_1v_0\\ v_1 & s_1v_1 & s_0v_0\end{array}$
\end{tabular}
\caption{Multiplication tables for the nine full $2\times2$ automata.}
\label{tables}
\end{table}

\emph{Aperiodicity.} By the periodicity dichotomy of Remark~\ref{rmk:aperiod}
together with a per-automaton analysis, $F_\theta$ is aperiodic only for the
lamplighter pair $\theta_4^a,\theta_4^c$; the remaining seven are periodic. Indeed
$\theta_1$ has $F_{\theta_1}\cong\NN^2$ with $C^*\cong C(\mathbb T^2)$, while
$\theta_2$ and $\theta_3$ have $C^*\cong M_2(C(\mathbb T))$, and
$\theta_5,\dots,\theta_8$ reduce to periodic systems. By
Proposition~\ref{prop:antitorus}, then, only $\theta_4^a$ and $\theta_4^c$ are
candidates for an anti-torus.

\emph{Bi-reversibility.} Both candidates survive. The lamplighter
$\theta_4^a$ has each $\lambda(v_i,\cdot)$ a bijection of $S$ (invertible) and each
$\pi(\cdot,s_j)$ a bijection of $V$ (reversible), and its inverse automaton is again
reversible; so $\theta_4^a$ is bi-reversible, equivalently four-way deterministic,
and likewise $\theta_4^c$. By contrast $\theta_3$ and $\theta_7$ are neither
invertible nor reversible, hence not four-way
deterministic; and $\theta_1,\theta_2,\theta_5,\theta_6,\theta_8$ are invertible and
reversible but periodic. Combining the two filters:
\[
\renewcommand{\arraystretch}{1.2}
\begin{array}{l|ccc}
\text{automaton} & \text{inv.\ \& rev.} & F_\theta\ \text{aperiodic} &
  \text{anti-torus}\\\hline
\theta_1 & \checkmark & - & \text{no}\\
\theta_2 & \checkmark & - & \text{no}\\
\theta_3 & \times & - & \text{no}\\
\theta_4^a & \checkmark & \checkmark & \textbf{yes}\\
\theta_4^c & \checkmark & \checkmark & \textbf{yes}\\
\theta_5 & \checkmark & - & \text{no}\\
\theta_6 & \checkmark & - & \text{no}\\
\theta_7 & \times & - & \text{no}\\
\theta_8 & \checkmark & - & \text{no}
\end{array}
\]

\noindent
A dash marks a periodic automaton; for those an anti-torus is precluded by
Proposition~\ref{prop:antitorus} regardless of the square-complex question. For
$\theta_4^a,\theta_4^c$ both filters are passed, so Theorem~\ref{thm:main} yields an
anti-torus.

\emph{Conclusion.} Exactly two full $2\times2$ Mealy automata carry an anti-torus:
the lamplighter pair $\theta_4^a,\theta_4^c$, bi-reversible and aperiodic, with
self-similar group $\ZZ\wr(\ZZ/2\ZZ)$ (see \cite{GNS}). The other seven each fail a hypothesis ---
$\theta_3,\theta_7$ are not four-way deterministic, while
$\theta_1,\theta_2,\theta_5,\theta_6,\theta_8$ are bi-reversible but periodic, with
finite self-similar groups ($\{1\}$ for $\theta_1,\theta_2$; $\ZZ/2\ZZ$ for
$\theta_5,\theta_6,\theta_8$). The free, irreducible behaviour --- two-sided paths
carrying the free group $\mathbb{F}_3$ --- first appears one alphabet-size up, at
$m=2$, $n=3$, in Example~\ref{ex:53}.

\begin{remark}[On the lamplighter]\label{rem:lamp-app}
The aperiodic $2\times2$ anti-torus is the lamplighter. Its self-similar group $\ZZ\wr(\ZZ/2\ZZ)$ is
amenable, but this is a property of the automaton group $G(A)$, not of $\pi_1(Y_A)$,
and so is no obstruction: the anti-torus is a period-free plane in the sense of
Lemma~\ref{lem:config}, a statement about the two-sided path space of $F_\theta$
rather than about $G(A)$. 
\end{remark}

\section{The \texorpdfstring{$3\times2$}{3x2} anti-torus}\label{app:3x2}

Beyond the nine $2\times2$ automata, we record a full automaton with $|V|=3$,
$|S|=2$ that carries an anti-torus; it is the transpose of the worked
Example~\ref{ex:53}. Write $V_3=\{v_0,v_1,v_2\}$ and $S_2=\{s_0,s_1\}$, with
degree-$\eps_1$ generators $f_0,f_1,f_2$ and degree-$\eps_2$ generators $g_0,g_1$.
The transition and output functions are
\[
\begin{array}{c|cc}
\pi & s_0 & s_1\\\hline
v_0 & v_1 & v_2\\
v_1 & v_0 & v_0\\
v_2 & v_2 & v_1
\end{array}
\qquad\qquad
\begin{array}{c|cc}
\lambda & s_0 & s_1\\\hline
v_0 & s_1 & s_0\\
v_1 & s_0 & s_1\\
v_2 & s_1 & s_0
\end{array}
\]
giving the multiplication table (the cell in row $f_i$, column $g_j$ is the
right-hand side $g_{j'}f_{i'}$ of $f_ig_j=g_{j'}f_{i'}$):
\[
\renewcommand{\arraystretch}{1.2}
\begin{array}{c|cc}
 & g_0 & g_1\\\hline
f_0 & g_1f_1 & g_0f_2\\
f_1 & g_0f_0 & g_1f_0\\
f_2 & g_1f_2 & g_0f_1
\end{array}
\]

Here $\lambda(v_0,\cdot)=\lambda(v_2,\cdot)=\sigma$ and
$\lambda(v_1,\cdot)=\mathrm{id}$, so $A$ is invertible; $\pi(\cdot,s_0)=(v_0\,v_1)$
and $\pi(\cdot,s_1)=(v_0\,v_2\,v_1)$ are permutations of $V_3$, so $A$ is reversible;
and all four corner maps of the six Wang tiles are bijections, so $A$ is four-way
deterministic, hence bi-reversible, and $Y$ is a complete square complex
(Proposition~\ref{prop:csc}). 
\begin{equation} \label{eq:anti-torus2}
\begin{array}{l}
\begin{tikzpicture}[>=stealth]
\draw (0,0)--(1,0)--(1,1)--(0,1)--(0,0)--(1,1)--(0,1)--(1,0);
\draw (2,0)--(3,0)--(3,1)--(2,1)--(2,0)--(3,1)--(2,1)--(3,0);
\draw (4,0)--(5,0)--(5,1)--(4,1)--(4,0)--(5,1)--(4,1)--(5,0);
\draw (6,0)--(7,0)--(7,1)--(6,1)--(6,0)--(7,1)--(6,1)--(7,0);
\draw (8,0)--(9,0)--(9,1)--(8,1)--(8,0)--(9,1)--(8,1)--(9,0);
\draw (10,0)--(11,0)--(11,1)--(10,1)--(10,0)--(11,1)--(10,1)--(11,0);
% T_{(f0,g0)}
\node at (0.5,0.85) {\tiny $g_1$};
\node at (0.5,0.15) {\tiny $g_0$};
\node at (0.2,0.5) {\tiny $f_1$};
\node at (0.8,0.5) {\tiny $f_0$};
% T_{(f0,g1)}
\node at (2.5,0.85) {\tiny $g_0$};
\node at (2.5,0.15) {\tiny $g_1$};
\node at (2.2,0.5) {\tiny $f_2$};
\node at (2.8,0.5) {\tiny $f_0$};
% T_{(f1,g0)}
\node at (4.5,0.85) {\tiny $g_0$};
\node at (4.5,0.15) {\tiny $g_0$};
\node at (4.2,0.5) {\tiny $f_0$};
\node at (4.8,0.5) {\tiny $f_1$};
% T_{(f1,g1)}
\node at (6.5,0.85) {\tiny $g_1$};
\node at (6.5,0.15) {\tiny $g_1$};
\node at (6.2,0.5) {\tiny $f_0$};
\node at (6.8,0.5) {\tiny $f_1$};
% T_{(f2,g0)}
\node at (8.5,0.85) {\tiny $g_1$};
\node at (8.5,0.15) {\tiny $g_0$};
\node at (8.2,0.5) {\tiny $f_2$};
\node at (8.8,0.5) {\tiny $f_2$};
% T_{(f2,g1)}
\node at (10.5,0.85) {\tiny $g_0$};
\node at (10.5,0.15) {\tiny $g_1$};
\node at (10.2,0.5) {\tiny $f_1$};
\node at (10.8,0.5) {\tiny $f_2$};
\end{tikzpicture}
\end{array}
\end{equation}

The single-vertex rank-two graph associated with this data is aperiodic. From the
$\pi$-table, the state permutations $\beta_{g_j}\colon V_3\to V_3$,
$\beta_{g_j}(f_i)=\pi(v_i,s_j)$, are $\beta_{g_0}=(f_0\,f_1)$ and
$\beta_{g_1}=(f_0\,f_2\,f_1)$. Taking $A=\{f_0,f_1\}$ and the one-letter word $g_0$,
we have $\beta_{g_0}(A)=A$, so $F$ is aperiodic by Proposition~\ref{prp:aperiod}.

The self-similar group associated to $A$ is the free group $\mathbb{F}_3$ --- the automaton
$G(2240)$ in the \cite{BGKMNSS} classification of $3$-state $2$-letter automata. By
Theorem~\ref{thm:main}, $Y$ contains an anti-torus.

Unlike the $2\times2$ lamplighter, this anti-torus is loop-spanned in the sense of
\S\ref{sec:loops}: its tiling has both axes constant (one carrying $f_0$, the other $g_0$) with
aperiodic interior, so both axes are loop axes and $(f_0,g_0)$ is a Wise anti-torus.
It is the smallest full automaton carrying one.

Moreover $\beta_{g_1}=(f_0\,f_2\,f_1)$ is a fixed-point-free $3$-cycle, so by
\cite[Lemma~3]{Rat} the centralizer $Z_{\pi_1(Y)}(g_1)=\langle g_1\rangle\cong\ZZ$
is cyclic; and since $Y$ carries a Wise anti-torus, $\pi_1(Y)$ is an irreducible
$(6,4)$--group (Corollary~\ref{cor:irred}, via \cite[Proposition~9]{Rat}). The
transposed Example~\ref{ex:53} is the corresponding irreducible $(4,6)$--group, the
degree of Wise's original anti-torus \cite{W}.

\end{document}